\documentclass[12pt,a4paper,final]{article}
\pdfoutput=1
\usepackage{graphicx}
\graphicspath{{./Figures/}}
\usepackage{subfigure}
\usepackage{authblk}

\usepackage{amssymb,latexsym}

\usepackage{setspace}

\usepackage{amsmath}

\usepackage{mathrsfs,amsfonts}

\usepackage{amsthm,amsxtra}

\usepackage[final]{showkeys}

\usepackage{stmaryrd}

\usepackage{algorithm}
\usepackage{algorithmicx}
\usepackage{algpseudocode}
\usepackage{mathtools}
\newif\ifPDF
\ifx\pdfoutput\undefined
\PDFfalse
\else
\ifnum\pdfoutput > 0
\PDFtrue
\else
\PDFfalse
\fi
\fi

\ifPDF
\usepackage{pdftricks}
\begin{psinputs}
	\usepackage{pstricks}
	\usepackage{pstcol}
	\usepackage{pst-plot}
	\usepackage{pst-tree}
	\usepackage{pst-eps}
	\usepackage{multido}
	\usepackage{pst-node}
	\usepackage{pst-eps}
\end{psinputs}
\else
\usepackage{pstricks}
\fi

\ifPDF
\usepackage[debug,pdftex,colorlinks=true, 
linkcolor=blue, bookmarksopen=false,
plainpages=false,pdfpagelabels]{hyperref}
\else
\usepackage[dvips]{hyperref}
\fi

\usepackage[openbib]{currvita}

\usepackage{fancyhdr}

\newtheorem{theorem}{Theorem}[section]
\newtheorem{lemma}[theorem]{Lemma}

\newtheorem{proposition}[theorem]{Proposition} 
\newtheorem{remark}[theorem]{Remark} 
\newtheorem{corollary}[theorem]{Corollary}

\newcommand{\supp}{\operatorname{supp}}

\newcommand{\pa}{\partial}
\newcommand{\rarrow}{\rightarrow}
\newcommand{\dis}{\displaystyle}
\newcommand{\disp}{\displaystyle}
\newcommand{\be}{\begin{equation}}
\newcommand{\ee}{\end{equation}}

\newenvironment{keywords}
{\noindent{\bf Key words.}\small}{\par\vspace{1ex}}

\makeatletter
\newcommand{\chapterauthor}[1]{%
	{\parindent0pt\vspace*{-25pt}%
		\linespread{1.1}\large\scshape#1%
		\par\nobreak\vspace*{35pt}}
	\@afterheading%
}
\makeatother
\title{An Inverse Hyperbolic Problem with Application to Joint Photoacoustic Parameter Determination}

\author[1]{Shitao Liu \thanks{liul@clemson.edu }}
\author[2]{Gunther Uhlmann
\thanks{gunther@math.washington.edu }}
\author[3]{Yang Yang
\thanks{yangy5@msu.edu}}
\affil[1]{School of Mathematical and Statistical Sciences, Clemson University}
\affil[2]{Department of Mathematics, University of Washington}
\affil[3]{Department of Computational Mathematics, Science and Engineering, Michigan State University}

\date{}

\begin{document}
\maketitle
\begin{abstract}

We consider an inverse problem of recovering a parameter appearing in all levels in a second-order hyperbolic equation from a single boundary measurement. The model is motivated from applications in photoacoustic tomography when one seeks to recover both the wave speed and the initial ultrasound pressure from a single ultrasound signal. In particular, our result shows that the ratio of the initial ultrasound pressure and the wave speed squared uniquely determines both of them respectively. 

\end{abstract}

\section{Introduction and Problem Formulation}
Let $\Omega\subset\mathbb{R}^n$, $n\geq 2$, be an open bounded domain with smooth enough (e.g., $C^2$) boundary $\Gamma$. We consider the following second-order hyperbolic equation for $u=u(x,t)$ defined on $Q=\Omega\times(0,T)$, along with initial position $f(x)$, initial velocity $0$ on $\Omega$ and the Dirichlet boundary condition $h$ on $\Sigma=\Gamma\times(0,T)$ that are given in appropriate function spaces:
\begin{equation}\label{1}
\begin{cases}
 u_{tt} - D(x)\Delta u - 2\nabla D(x) \cdot \nabla u - \Delta D(x) u = 0 & \mbox{in } Q \\[2mm]
 u\left(x,0\right) = f(x); \ u_t\left(x,0\right) = 0 & \mbox{in } \Omega \\[2mm]
 u(x,t) = h(x,t) & \mbox{in } \Sigma.
\end{cases}
\end{equation}
Here $D=D(x)$ is a smooth enough function (e.g. $C^2$) 
, $\frac{1}{c_0}\leq D\leq c_0$ for some $c_0>0$, such that $D-1$ is compactly supported in $\Omega$. 

We assume the boundary $\Gamma=\pa\Omega=\overline{\Gamma_0\cup\Gamma_1}$, where $\Gamma_0\cap\Gamma_1=\emptyset$. Here $\Gamma_1$ is referred as the observed part of the boundary where the measurements are taken, and $\Gamma_0$ is referred as the unobserved part of the boundary where we do not have access to make measurements.
In this paper, we are interested in the following inverse problem: Recover the coefficient $D(x)$ from a {\it single} measurement of $\partial_{\nu} u|_{\Gamma_1\times(0,T)}$, i.e., the Neumann boundary trace of the solution $u$ over the observed boundary $\Gamma_1$ over the time interval $(0,T)$. Here $T$ should be sufficiently large due to the finite propagation speed of the system \eqref{1}. 

To make the observed part $\Gamma_1$ of the boundary more precise, in this paper we assume the following standard geometrical assumptions on the domain $\Omega$ and the unobserved part of the boundary $\Gamma_0$:

\smallskip
(A.1) There exists a strictly convex function $d: \overline{\Omega} \rarrow \mathbb{R}$ in the metric $g=D^{-1}(x)dx^2$, and of class $C^3(\overline{\Omega})$, such that the following two properties hold true (through translation and re-scaling if necessary):

\smallskip
(i) The normal derivative of $d$ on the unobserved part $\Gamma_0$ of the boundary is non-positive. Namely,
\begin{equation*}
\frac{\pa d}{\pa\nu}=\langle \mathcal{D}d(x), \nu(x)\rangle\leq 0, \quad  \forall x\in\Gamma_0,
\end{equation*}
where $\mathcal{D}d=\nabla_g d$ for the scalar function $d$ is the gradient vector field on $\Omega$ with respect to the metric $g$.

\medskip
(ii) \begin{equation*} \mathcal{D}^2d(X,X) = \langle \mathcal{D}_X(\mathcal{D}d), X\rangle_g\geq 2|X|_g^2, \ \forall X\in \mathcal{M}_x, \ \min_{x\in\overline{\Omega}}d(x)=m_0>0\end{equation*}
where $\mathcal{D}^2d$ is the Hessian of $d$ (a second-order tensor) and $\mathcal{M}_x$ is the tangent space at $x\in\Omega$.

\smallskip
(A.2) $d(x)$ has no critical point on $\overline{\Omega}$. In other words,
\begin{equation*}
\inf_{x\in\overline{\Omega}}|\mathcal{D}d|>0, \ \mbox{so that we may take} \ \inf_{x\in\overline{\Omega}}\frac{|\mathcal{D}d|^2}{d}>4.
\end{equation*}
\begin{remark}\label{rem2}
The geometrical assumptions above permit the construction of a vector field that enables a pseudo-convex function necessary for allowing a Carleman estimate containing no lower-order terms for the second-order hyperbolic equation (\ref{1}). In fact, this can be done for general second-order hyperbolic equations and metric (see Section 2). These type assumptions were formulated as early as in H\"{o}rmander's classical work on partial differential operators of principal types \cite{H1963}. 
Specifically they were also formulated in 
\cite{TY2002} 
under the more general Riemannian geometry framework.  For examples and detailed illustrations of large general classes of domains and voundaries $\{\Omega, \Gamma_1, \Gamma_0\}$ satisfying the aforementioned assumptions we refer to \cite[Appendix B]{TY2002}. 
\end{remark}

\medskip
Now let us state the main theorem regarding our inverse problem.
\begin{theorem}\label{Th1}
Under the geometrical assumptions (A.1) and (A.2) and let 
\begin{equation}\label{largetime}
T > T_0 := 2\disp\sqrt{\max_{x\in\overline{\Omega}}d(x)}
\end{equation} 
where $d(x)$ is from the geometrical assumption (A.1).
Denote by $u_1$ and $u_2$ the corresponding solutions of the equation (\ref{1}) with different coefficients $D_1$ and $D_2$, respectively. 

Suppose the initial and boundary conditions are in the following function spaces
\begin{equation}\label{regularity}
f\in H^{\gamma}(\Omega), \ h\in H^{{\gamma}}(\Sigma), \ \mbox{where} \ \gamma>\frac{n}{2}+3
\end{equation}
along with all compatibility conditions (trace coincidence) which make sense. 
In addition, suppose the following positivity condition holds: There exists a compact set $K\subset\Omega$ with $C^2$-boundary such that $\text{supp } (D_1-D_2)\subset {\rm Int } K$ (Here ${\rm Int } K$ denotes the interior of $K$) and there exists a positive number $r_0>0$ such that 
\begin{equation}\label{positivity}
|f(x)|\geq r_0 > 0, \ a.e.\ x\in K.
\end{equation}
Then there exists a constant $C = C(\Omega, \Gamma, T, d, c_0)>0$ such that 
\begin{equation}
\|D_1-D_2\|_{\Delta} \leq C \left\|\partial_t\partial_{\nu}u_1-\partial_t\partial_{\nu}u_2\right\|_{L^2(\Gamma_1\times(0,T))}.
\end{equation}
Here the norm $\|\cdot\|_{\Delta}$ is defined by 
\begin{equation}
\|D\|_{\Delta}^2:= \int_{\Omega} \left(|\Delta D|^2 + |\nabla D|^2 + |D|^2\right)\, d\Omega.
\end{equation}
\end{theorem}
 
\noindent{\bf Inverse source problem}. The first step to solve this inverse problem is to convert it into a corresponding {\it inverse source problem}. Indeed, suppose we have two coefficients $D_1(x)$ and $D_2(x)$ with the corresponding solutions $u_1(x,t)$ and $u_2(x,t)$. If we let 
\begin{equation}\label{relation}
F(x) = D_1(x) - D_2(x), \ w(x,t) = u_1(x,t) - u_2(x,t), \ R(x,t) = u_2(x,t)
\end{equation}
then $w=w(x,t)$ is readily seen to satisfy the following mixed problem with homogeneous initial and boundary conditions:
\begin{equation}\label{2}
{
\begin{cases}
w_{tt} - D_1(x)\Delta w - 2\nabla D_1(x)\cdot\nabla w - \Delta D_1(x) w
=S(x,t)
& \mbox{in } Q\\[2mm]
w\left(x,0\right) = w_t\left(x,0\right) = 0 & \mbox{in } \Omega \\[2mm]
w(x,t) = 0 & \mbox{in } \Sigma
\end{cases}
}
\end{equation}
where for convenience we denote on the right-hand side
\begin{equation*}\label{rhs}
    S(x,t) = \Delta F(x) R(x,t) + 2\nabla F(x)\cdot\nabla R(x,t) + F(x)\Delta R(x,t).
\end{equation*}
Here we assume that $D_1(x)$ and $R=R(x,t)$ are given fixed functions, and the coefficient $F(x)$ in the source term $S(x,t)$ is unknown. Then the inverse source problem is to determine $F(x)$ from the Neumann boundary measurement of $w$ over the observed part $\Gamma_1$ of the boundary and over a sufficiently long time interval $(0, T)$.  
More specifically, corresponding with Theorem \ref{Th1}, we will prove the following stability result for the inverse source problem. 

\begin{theorem}\label{Th2}
Under the geometrical assumptions (A.1) and (A.2) and let $T$ satisfy (\ref{largetime}). 
Assume the following regularity and positivity conditions:
\begin{equation}\label{regularity2}
R\in W^{3,\infty}(Q)
\end{equation}
and there exists $r_0>0$ such that 
\begin{equation}\label{positivity1}
|R(x,0)|\geq r_0, \ a.e. \ x\in K. 
\end{equation}
Then there exists $C = C(\Omega, \Gamma, T, R, d, c_0)>0$ such that 
\begin{equation}
\|F\|_{\Delta} \leq C \left\|\partial_t\partial_{\nu}w\right\|_{L^2(\Gamma_1\times(0,T))}.
\end{equation}
\end{theorem}

\textbf{Motivation and Application to Photoacoustic Tomography.} The problem studied in this paper is motivated by an inverse problem arising in the imaging modality of \textit{Photo-Acoustic Tomography (PAT)}. 
PAT illuminates biological tissue with pulsed laser to generates ultrasound through thermo-elastic expansion, and various models have been proposed in the literature based on different experimental setups (see e.g, \cite{acosta2015multiwave,acosta2018thermoacoustic,chervova2016time,cox2007photoacoustic,homan2013multiwave,nguyen2016dissipative,palacios2016reconstruction,stefanov2015multiwave,stefanov2017multiwave,Stefanov2017planar}). In this paper, we are interested in the following: Suppose propagation of the ultrasound in the space is modeled by the initial value problem:
\begin{equation} \label{eq:p}
\begin{cases}
 p_{tt} - c^2(x) \Delta p = 0 & \mbox{ in } \mathbb{R}^n\times (0,T) \\[2mm]
 p\left(x,0\right) = p_0(x); \ p_t\left(x,0\right) = 0 & \mbox{ in } \mathbb{R}^n. \\[2mm]
\end{cases}
\end{equation}
Here $p=p(x,t)$ is the ultrasound pressure, $c(x)>0$ is the wave speed, $p_0(x)$ is the initial ultrasound pressure induced by laser illumination, and $T>0$ is the duration of measurement. 
Let the open bounded domain $\Omega$ represent the biological tissue. It is typically assumed that (1) $c \in C^\infty(\mathbb{R}^n)$ is non-trapping so that all ultrasound signals are detectable from boundary measure, and (2) $p_0$ and $c-1$ are compactly support inside $\Omega$ so that the inhomogeneity exists only inside the tissue. The imaging problem in PAT aims to recover the initial pressure $p_0(x)$ from the measurement of the boundary ultrasound signal $p|_{\partial\Omega\times(0,T)}$.

When the sound speed $c(x)>0$ is known, the imaging problem in PAT is a linear inverse problem and has been extensively studied and well understood. For instance, see~\cite{agranovsky2009on, finch2004determining, finch2009recovering, haltmeier2005filtered, hristova2009time, hristova2008reconstruction, kuchment2008mathematics, kunyansky2007explicit, qian2011an, Stefanov2009variable, Stefanov2011brain, Stefanov2017planar} and the reference therein. 
In contrast, when the sound speed $c(x)>0$ is unknown, the imaging problem results in a nonlinear inverse problem due to the nonlinear dependence on $c$ of the data. In this case, one usually attempts to jointly recover $c$ and $p_0$.
Theoretical results on joint identifiability of both $c$ and $p_0$ are relatively limited~\cite{huang2023piecewise, kian2023determination, knox2018jointdetermination, liu2015jointdetermination, stefanov2013instability}, and additional assumptions are usually required to prove the uniqueness.
Among these results, the combination $p_0 c^{-2}$ arises naturally in the analysis, and is shown to be uniquely identifiable in several cases including: (1) $p_0 c^{-2}$ depends on one less spatial variable~\cite{liu2015jointdetermination}; (2) $p_0 c^{-2}$ is piecewise constant on a disjoint union of special convex domains with known amplitudes~\cite{huang2023piecewise}. 
These facts motivate the following question: \textit{Is the knowledge of the combination $p_0 c^{-2}$ sufficient to determine both $c$ and $p_0$?}

The inverse problem addressed in this paper provides an affirmative answer to this question. Indeed, if we take $\Gamma_1=\partial\Omega$ and $\Gamma_0=\emptyset$ to match the full boundary measurement in PAT, and take
\begin{equation} \label{eq:variable}
u(x,t) = p(x,t) c^{-2}(x), \quad D(x) = c^2(x), \quad f(x) = p_0(x) c^{-2}(x), \quad h(x,t) = p(x,t)|_{\Sigma}.
\end{equation}
Using~\eqref{eq:p}, it is easy to verify that inside $\Omega$, $u$ satisfies the boundary value problem~\eqref{1}. On the other hand, $u=u(x,t)$ satisfies the following exterior boundary value problem outside $\Omega$ (since $p_0$ and $c-1$ are compactly supported in $\Omega$):
\begin{equation} \label{eq:ext}
\begin{cases}
 u_{tt} - \Delta u = 0 & \mbox{in } \mathbb{R}^n\times (0,T)\backslash \overline{Q} \\[2mm]
 u\left(x,0\right) = 0; \ u_t\left(x,0\right) = 0 & \mbox{in } \Omega \\[2mm]
 u(x,t) = h(x,t) & \mbox{in } \Sigma.
\end{cases}
\end{equation}
One can solve this exterior problem to obtain the Neumann boundary trace $\partial_\nu u|_\Sigma$ to the interior problem~\eqref{1}.

Theorem~\ref{Th1} can be applied to obtain novel joint recovery outcomes in PAT when combined with the determination result of $p_0 c^{-2}$. For instance, it is proved that $p_0 c^{-2}$ can be uniquely determined in 3D~\cite{liu2015jointdetermination} and 2D~\cite{huang2023piecewise} if this combination is independent of one of the spatial variables $x_1,\dots,x_n$ and if $T=\infty$. We can strengthen these results to obtain the following uniqueness:
\begin{corollary} \label{thm:PAT}
    Let $n=2 \text{ or } 3$. 
    Suppose 
    \begin{enumerate}
        \item $p_0,\tilde{p}_0 \in H^{\gamma}(\Omega)$ with $\gamma>\frac{n}{2}+3$ are compactly supported in $\Omega$, and $p_0\geq r_0, \tilde{p}_0 \geq r_0$ on $K$ for some $r_0>0$;
        \item $c, \tilde{c}\in C^\infty(\mathbb{R}^n)$ are strictly positive, $c-1, \tilde{c}-1$ are compactly supported in $\Omega$, and the geometrical assumptions (A.1) and (A.2) hold for both $D=c^2$ and $\tilde{D}=\tilde{c}^2$.
    \end{enumerate}
    Suppose the a-priori relation $p_{0} c^{-2} = \tilde{p}_{0} \tilde{c}^{-2}$ holds. If $p|_{\Sigma} = \tilde{p}|_{\Sigma}$, then $c=\tilde{c}$ and $p_0=\tilde{p}_0$.
    
\end{corollary}

\textbf{Brief Literature on Inverse Hyperbolic Problems}.
Recovering coefficients in a second-order hyperbolic equation from making appropriate boundary measurements is an important type of inverse problems and it is often referred as an {\it inverse hyperbolic problem}. Such problems have been studied extensively since the 1980s and here we only mention the monographs and lecture notes \cite{B-Y2017,Isakov2006,Klibanov2021,Klibanov2004,LRS1986,Liu-T2013} and refer to the substantial lists of references therein.
The classical inverse hyperbolic problems usually involve recovering a single unknown coefficient, typically the damping coefficient or the potential coefficient, from a {\it single} boundary measurement of the solution \cite{BCIY2001,Imanuvilov2001a,Imanuvilov2003,LiuT2011-2,Liu2012,Y1999}. It is also possible to recover the variable wave speed, which would require the use of Riemannian geometry as the unknown wave speed is at the principal order level, through a single boundary measurement \cite{B2004, Liu2013, stefanov2013recovery}. In the case of a vector-valued unknown gradient coefficient, one may recover it by properly making $n$ sets of boundary measurements \cite{Jellali2006}. In the recent work \cite{LPS2023}, it was shown that all the aforementioned coefficients can be recovered all together at once by appropriately choosing finitely many initial conditions and measuring the corresponding boundary data.    

In our present model, the unknown coefficients appear in front of the principal order, first order and zeroth order level terms. As in \cite{LPS2023}, it is possible to recover those coefficients from finitely many boundary measurements.  
Nevertheless, since the unknown coefficients are all coming from the same function $D(x)$, one may still expect to be able to recover the function from just a single boundary measurement. A similar example to our inverse problem is to recover the coefficient $a(x)$ in the acoustic wave equation $u_{tt}=div \left(a(x)\nabla u\right)$, where the unknown coefficient $a(x)$ is involved together with its gradient. Indeed, \cite{Imanuvilov2003} and \cite{KY2006} proved the Lipschitz stability of recovering $a(x)$, from only a single boundary measurement when $a(x)$ satisfies appropriate conditions.

The standard approach for single measurement type inverse hyperbolic problems typically requires using the Carleman-type estimates for the underlying second-order hyperbolic equations. To certain extent, all such approaches can be seen as variations or improvements of the so called Bukhgeim--Klibanov (BK) method which was originated in the seminal paper \cite{BK1981}, see also \cite{Klibanov1992}. 
Our approach to solve the present inverse problem uses a more recent variation of the BK method that combines 
a Carleman estimate for general second-order hyperbolic equations \cite{TY2002} and a post Carleman estimate route introduced in \cite{HIY2020}. In particular, a novelty in our proof is that we also use a Carleman estimate for the second-order elliptic equation in the process (see Proposition~3.1 below) due to the specific structure of our model.

The rest of the paper is organized as follows. In the next section we recall some necessary tools for general second-order elliptic and hyperbolic equations to solve the inverse problem. 
In Section 3 we provide the proofs of Theorems \ref{Th1} and \ref{Th2}. 

\section{Some Preliminaries}

In this section we recall some key ingredients of the proofs used in the next section. This includes the Carleman estimates for the general second-order elliptic and hyperbolic equations defined on a Riemannian manifold, as well as the standard {\it a-priori} energy estimates and regularity theory for the general second-order hyperbolic equation with Dirichlet boundary condition. 
For simplicity here we only state the main results and refer to \cite{LLT1986,LM1972,TX2007,TY2002} for greater details.

To begin with, consider a Riemannian metric $g(\cdot,\cdot)=\langle \cdot, \cdot \rangle$ and squared norm $|X|^2=g(X,X),$ on a smooth finite dimensional manifold $M$. On the Riemannian manifold $(M, g)$ we define $\Omega$ as an open bounded, connected set of $M$ with smooth boundary $\Gamma=\overline{\Gamma_0\cup\Gamma_1}$, where $\Gamma_0\cap\Gamma_1=\emptyset$.  Let $\nu$ denote the unit outward normal field along the boundary $\Gamma$.  Furthermore, we denote by $\Delta_g$ the Laplace--Beltrami operator on the manifold $M$ and by $\mathcal{D}$ the Levi--Civita connection on $M$.  

Consider the following second-order elliptic equation with energy level terms defined on $\Omega$:
\be\label{mainequation2}
\Delta_g u(x,t) + F(u) = G(x), \quad x\in \Omega
\ee
where the forcing term $G\in L^2(\Omega)$ and the energy level term $F(u)$ is given by
\begin{equation*}\label{fw}
F(u)= \langle{\bf P}(x), \mathcal{D}u\rangle+P_0(x)u
\end{equation*}
with $P_0$ being a function and ${\bf P}$ being a vector field on $\Omega$ that satisfy there exists a constant $C>0$ such that
\begin{equation*}
|F(u)|\leq C[u^2+|\mathcal{D}u|^2], \ \forall x\in\Omega.
\end{equation*}

\noindent\textbf{Carleman estimate for the general second-order elliptic equations}. Under the geometrical assumptions (A.1) and (A.2), which guarantees the existence of the strict convex function $d$ with the general metric $g$. For a solution of the above equation (\ref{mainequation2}) that satisfies $u\in H^1(\Omega)$ and $\langle\mathcal{D}u, \nu\rangle\in L^2(\partial\Omega)$, and for arbitrary $\epsilon>0$ and $0<\delta_0<1$, we have the following Carleman estimate \cite[Corollary 4.2]{TX2007}:

\begin{equation}\label{elliptic}
K_{1,\tau} \int_{\Omega}e^{2\tau d} |\mathcal{D} u|^2\,d\Omega + K_{2,\tau} \int_{\Omega}e^{2\tau d} |u|^2\,d\Omega \leq \left(1+\frac{1}{\epsilon}\right) \int_{\Omega}e^{2\tau d} |G|^2\,d\Omega + BT(u).
\end{equation}
where the constants $K_{1,\tau}$ and $K_{2,\tau}$ are given by
\begin{equation}\label{Ks}
K_{1,\tau} = \delta_0 \left(2\rho\tau - \frac{\epsilon}{2}\right), \quad\quad K_{2,\tau} = 4\rho k^2\tau^3(1-\delta_0) + O(\tau^2).
\end{equation}
In addition, the boundary terms $BT(u)\equiv 0$ whenever the Cauchy data of $u$ vanish on the boundary $\partial\Omega$, i.e., $u=\langle\mathcal{D}u, \nu\rangle=0$ on $\partial\Omega$.

Next let us consider the second-order hyperbolic equation with energy level terms defined on $Q_T = \Omega\times(-T,T)$ for some $T>0$:
\be\label{mainequation}
w_{tt}(x,t) - \Delta_g w(x,t) + \widetilde{P}(w) = \widetilde{G}(x,t), \quad (x, t)\in Q_T
\ee
where $\widetilde{G}\in L^2(Q_T)$ and $\widetilde{P}(w)$ is given by
\begin{equation}\label{fw}
\widetilde{P}(w)= \langle{\bf \widetilde{P}}(x,t), \mathcal{D}w\rangle +\widetilde{P}_1(x,t)w_t+\widetilde{P}_0(x,t)w.
\end{equation}
Here $\widetilde{P}_0$, $\widetilde{P}_1$ are functions on $Q_T$, ${\bf \widetilde{P}}$ is a vector field on $\Omega$ for $t\in(-T,T)$, and they satisfy the following estimate: There exists a constant $C_T>0$ such that 
\begin{equation}\label{fwproperty}
|\widetilde{P}(w)|\leq C_T[w^2+w_t^2+|\mathcal{D}w|^2], \ \forall(x, t)\in Q_T.
\end{equation}

\noindent\textbf{Carleman estimate for general second-order hyperbolic equations}. 
Having chosen the strictly convex function $d(x)$ as in the geometric assumption (A.1) with respect to a general metric $g$, we can define the function $\varphi(x,t): \Omega\times\mathbb{R}\to\mathbb{R}$ of class $C^3$
by setting
\begin{equation*}\label{pseudo1}
\dis\varphi(x,t)=d(x)-ct^2, \quad x\in\Omega, \  t\in(-T,T),
\end{equation*}
where $T>T_0$ as in (\ref{largetime}). Moreover, $c\in(0,1)$ is selected as follows: Let $T>T_0$ be given, then there exists $\delta>0$ such that
\begin{equation*}\label{pseudo3}
\dis T^2>4\max_{x\in\overline{\Omega}}d(x)+4\delta.
\end{equation*}
For this $\delta>0$, there exists a constant $c\in(0,1)$, such that
\begin{equation*}\label{pseudo4}
\dis cT^2>4\max_{x\in\overline{\Omega}}d(x)+4\delta.
\end{equation*}
It is easy to check such function $\varphi(x,t)$ carries the following properties:

\medskip
(a) For the constant $\delta>0$ fixed above, we have 
\begin{equation}\label{pseudoproperty1}
\dis\varphi(x,-T)=\varphi(x,T) \leq\max_{x\in\overline{\Omega}}d(x)-cT^2\leq-\delta
\ \text{uniformly in} \ x\in\Omega;
\end{equation}
and
\begin{equation}\label{pseudoproperty2}
\dis\varphi(x,t)\leq\varphi\left(x,0\right), \quad \text{for any} \ t\in(-T,T) \ \text{and any} \ x\in\Omega.
\end{equation}

(b) There are $t_0$ and $t_1$, with $-T<t_0<0<t_1<T$, say, chosen symmetrically about $0$, such that
\begin{equation}\label{pseudoproperty3}
\min_{x\in\overline{\Omega},t\in[t_0,t_1]}\varphi(x,t)\geq\sigma,
\quad \text{where} \ 0<\sigma<m_0 = \min_{x\in\overline{\Omega}}d(x).
\end{equation}
Moreover, let $Q(\sigma)$ be the subset of $Q_T$ defined by
\begin{equation}\label{qsigma}
Q{(\sigma)}=\{(x,t): \varphi(x,t)\geq\sigma>0, x\in\Omega, -T< t< T\}.
\end{equation}

We now return to the equation \eqref{mainequation}, and consider solutions $w(x,t)$ in the class
\begin{equation*}\label{h1}
\begin{cases}
w\in H^{1,1}(Q) = L^2((-T,T);H^1(\Omega))\cap H^1((-T,T);L^2(\Omega)); \\[2mm]
w_t, \frac{\pa w}{\pa\nu}\in L^2((-T, T); L^2(\Gamma)).
\end{cases}
\end{equation*}

Then for these solutions with geometrical assumptions (A.1) and (A.2) on $\Omega$, the following one-parameter family of estimates hold true, with $\beta>0$ being a suitable constant ($\beta$ is positive by virtue of (A.2)), for all $\tau>0$ sufficiently large and $\epsilon>0$ small \cite[Theorem 5.1]{TY2002}:
\begin{multline}\label{carleman}
BT(w)+2\int_{Q}e^{2\tau\varphi}|\widetilde{G}|^2\,dQ+C_{1,T}e^{2\tau\sigma}\int_{Q}w^2\,dQ+ c_T\tau^3e^{-2\tau\delta}[E_w(-T)+E_w(T)] \\[2mm]
\geq C_{1,\tau}\int_Qe^{2\tau\varphi}[w_t^2+|\mathcal{D}w|^2]\,dQ + C_{2,\tau}\int_{Q{(\sigma)}}e^{2\tau\varphi}w^2\,dxdt
\end{multline}
where
\begin{equation}\label{c1tauc2tau}
C_{1,\tau}=\tau\epsilon(1-\alpha)-2C_T, \quad C_{2,\tau}=2\tau^3\beta+\mathcal{O}(\tau^2)-2C_T.
\end{equation}
Here $\delta>0$, $\sigma>0$ are the constants as in above, $C_T$, $c_T$ and $C_{1,T}$ are positive constants depending on $T$, as well as $d$ (but not on $\tau$).
The energy function $E_w(t)$ is defined as
\begin{equation}\label{energy1}
E_w(t)=\int_{\Omega}[w^2(x,t)+w_t^2(x,t)+|\mathcal{D}w(x,t)|^2]\,d\Omega.
\end{equation}
In addition, $BT(w)$ stands for boundary terms and can be explicitly calculated as
\begin{eqnarray}\label{boundary}
BT(w) & = & 2\tau\int_{\Sigma}e^{2\tau\varphi}\left(w_t^2-|\mathcal{D}w|^2\right)\langle \mathcal{D}d,\nu\rangle\,d\Sigma \nonumber\\[2mm]
&+& 4\tau\int_{\Sigma}e^{2\tau\varphi}\langle \mathcal{D}d, \mathcal{D}w\rangle\langle \mathcal{D}w,\nu\rangle\,d\Sigma + 8\alpha\tau\int_{\Sigma}e^{2\tau\varphi}tw_t \langle \mathcal{D}w,\nu\rangle\,d\Sigma \nonumber\\[2mm]
&+& 4\tau^2\int_{\Sigma}e^{2\tau\varphi}\left[|Dd|^2-4\alpha^2t^2+\frac{\Delta d-\alpha-1}{2\tau}\right]w \langle \mathcal{D}w,\nu\rangle\,d\Sigma \nonumber \\[2mm]
&+& 2\tau\int_{\Sigma}e^{2\tau\varphi}\left[2\tau^2\left(|\mathcal{D}d|^2-4\alpha^2t^2\right)+ \tau(3\alpha+1)\right]w^2\langle \mathcal{D}d,\nu\rangle\,d\Sigma.
\end{eqnarray}

\noindent\textbf{A-priori energy estimate and regularity theory for general second-order hyperbolic equations with Dirichlet boundary condition}.
Consider the second-order hyperbolic equation (\ref{mainequation}) with initial conditions $w(x, 0)=w_0(x)$, $w_t(x,0)=w_1(x)$ and Dirichlet boundary condition $w(x,t) = h(x,t)$ on $\Sigma_T = \Gamma\times(-T,T)$. If $\widetilde{F}$ satisfies (\ref{fw}), (\ref{fwproperty}),
then the following {\it a-priori} estimate holds true for the solutions $w$: there exists $C = C(\Omega, \Gamma, T)>0$ such that
\begin{equation}\label{energyestimate}
E_w(t) \leq C \left(\left\|w_0\right\|_{H^1(\Omega)}^2 + \left\|w_1\right\|_{L^2(\Omega)}^2 + \left\|h\right\|_{H^1(\Sigma_T)} + \left\|\widetilde{G}\right\|_{L^2(Q_T)}^2\right), \ \forall t\in(-T, T)
\end{equation}
where the energy is defined as in (\ref{energy1}).

In addition, the following interior and boundary regularity results for the solution $w$ hold true:
For $\gamma\geq 0$ (not necessarily an integer), if the given data satisfies the following regularity assumptions 
\begin{equation*}
\begin{cases}
\widetilde{G}\in L^1(0,T; H^{\gamma}(\Omega)), \ \pa_t^{(\gamma)}\widetilde{G}\in L^1(0,T; L^2(\Omega)), \\[2mm]
w_0\in H^{{\gamma}+1}(\Omega), \ w_1\in H^{\gamma}(\Omega), \ h\in H^{{\gamma}+1}(\Sigma_T)
\end{cases}
\end{equation*}
with all compatibility conditions (trace coincidence) which make sense. Then, we have the following regularity for the solution $w$:
\be\label{reg}
w\in C([0,T]; H^{{\gamma}+1}(\Omega)), \ \pa_t^{({\gamma}+1)}w\in C([0,T]; L^2(\Omega)); \ \frac{\pa w}{\pa\nu}\in H^{\gamma}(\Sigma_T).
\ee




\section{Proof of Main Theorems}

In this section we give the main proofs of the results established in the first section. We focus on proving Theorem \ref{Th2} for the inverse source problem since Theorems \ref{Th1} of the original inverse problem will then follow from the relation \eqref{relation} between the two problems with the regularity theory of the second-order hyperbolic equations.
Henceforth for convenience we use $C$ to denote a generic positive constant which may depend on $\Omega$, $T$, $D$, $D_1$, $r_0$, $w$, $u$, $R$, but not on the free large parameter $\tau$ appearing in the Carleman estimates recalled in Section 2. Moreover, we simply denote $K_{1,\tau}$, $C_{1,\tau}$ as $\tau$, and $K_{2,\tau}$, $C_{2,\tau}$ as $\tau^3$, respectively, see (\ref{Ks}) and (\ref{c1tauc2tau}).

\medskip 
\noindent\textbf{Proof of Theorem~\ref{Th2}}. 
We return to the $w$-equation (\ref{2}), extend $w$ and $R$ as even functions to $(-T,0)$, and get  

\begin{equation}\label{w}
{
\begin{cases}
w_{tt} - D_1(x)\Delta w - 2\nabla D_1(x)\cdot\nabla w - \Delta D_1(x) w
=S(x,t)
& \mbox{in } Q_T = \Omega\times(-T,T)\\[2mm]
w\left(x,0\right) = w_t\left(x,0\right) = 0 & \mbox{in } \Omega \\[2mm]
w(x,t) = 0 & \mbox{in } \Sigma_T = \Gamma\times(-T,T)
\end{cases}
}
\end{equation}
where again we denote
\begin{equation}\label{rhs}
    S(x,t) = \Delta F(x) R(x,t) + 2\nabla F(x)\cdot\nabla R(x,t) + F(x)\Delta R(x,t).
\end{equation}
Differentiate the above system in time $t$, we get the $w_t$-equation
\begin{equation}\label{wt}
{
\begin{cases}
(w_t)_{tt} - D_1(x)\Delta w_t - 2\nabla D_1(x)\cdot\nabla w_t - \Delta D_1(x) w_t
=S_t(x,t)
& \mbox{in } Q_T\\[2mm]
w_t\left(x,0\right) =0, \ (w_t)_t\left(x,0\right) = S(x,0) & \mbox{in } \Omega \\[2mm]
w_t(x,t) = 0 & \mbox{in } \Sigma_T
\end{cases}
}
\end{equation}

Note since $D_1\in C^{2}(\Omega)$, the equation in (\ref{wt}) can be written as a Riemannian wave equation with respect to the metric $g={D_1^{-1}(x)}dx^2$, modulo lower-order terms
\begin{equation*}
(w_t)_{tt} - \Delta_g w_t + \mbox{``lower-order terms''} = S_t(x,t).
\end{equation*}
(More precisely, here $\disp \Delta_g w_t = D_1(x)\Delta w_t + D_1(x)^{\frac{n}{2}}\nabla(D_1^{\frac{2-n}{2}})\cdot\nabla w_t$).
Furthermore, by the regularity assumption \eqref{regularity2}, we have that $S_t\in L^2(Q)$ and $S(x,0)\in L^2(\Omega)$.
Thus we may apply the Carleman estimate \eqref{carleman} for solution $w_t$ and get the following inequality for sufficiently large $\tau$:
\begin{eqnarray}\label{th1carleman1}
& \ & \tau\int_{Q_T}e^{2\tau\varphi}[(w_{tt})^2+|\mathcal{D} w_t|^2]\,dtdx+\tau^3\int_{Q{(\sigma)}}e^{2\tau\varphi}(w_t)^2dtdx \nonumber\\[2mm]
& \leq & BT(w_t)+Ce^{2\tau\sigma}\int_{Q_T}w_t^2\,dtdx + C\tau^3e^{-2\tau\delta}[E(-T)+E(T)] \nonumber \\
& \ & + C\int_{Q_T}e^{2\tau\varphi}\left[(\Delta F)^2R_t^2+|\nabla F|^2|\nabla R_t|^2+F^2(\Delta R_t)^2\right]\,dtdx.
\end{eqnarray}
Here $\sigma>0$ and $\delta>0$ are parameters as in Section 2, the energy function is given by
\begin{equation*}\label{energy}
E(t)=\int_{\Omega}[w_t^2(x,t)+w_{tt}^2(x,t)+|\mathcal{D}w_t(x,t)|^2]\,d\Omega
\end{equation*}
and the gradient vector field is given by $\mathcal{D}w_t = D_1(x) \nabla w_t$ and
\begin{equation}\label{gradientnorm}
|\mathcal{D}w_t|^2 = \langle \mathcal{D}w_t, \mathcal{D}w_t\rangle_g = \sum_{i=1}^n\frac{1}{D_1(x)}\left(D_1(x)\frac{\partial w_t}{\partial x_i}\right)\left(D_1(x)\frac{\partial w_t}{\partial x_i}\right) = D_1(x)|\nabla w_t|^2.
    \end{equation}
In addition, note from (\ref{w}) and (\ref{rhs}) we have 
\begin{equation}\label{wtt0}
w_{tt}(x,0) = R(x,0) \Delta F(x) + 2\nabla R(x,0)\cdot\nabla F(x) + \Delta R(x,0) F(x).
\end{equation}
For this equation, by the assumption $|R(x,0)|\geq r_0>0$ a.e. on $K$, we have the following Carleman type estimate:
\begin{proposition}
With the assumption (\ref{positivity1}), there exists $C = C(\Omega, \Gamma, R, d)>0$ such that for any $F\in H^2_0(\Omega)$ satisfying (\ref{wtt0}) we have for all $\tau>0$ large enough
\begin{equation}\label{ellipticcarleman}
\int_{\Omega}e^{2\tau d} \left(|\Delta F|^2 + |\nabla F|^2 + |F|^2\right)\,dx \leq C \int_{\Omega} e^{2\tau d}w_{tt}^2(x,0)\,dx
\end{equation}
where $d=d(x)$ is the function from the geometrical assumption (A.1).
\end{proposition}

\begin{proof}
With the assumption (\ref{positivity1}), we may divide the equation (\ref{wtt0}) throughout by $R(x,0)$. Moreover, we write the Euclidean Laplacian $\Delta$ as the Laplace--Beltrami operator $\Delta_g$ with respect to the metric $g = D_1^{-1}(x)dx^2$ with the aforementioned formula
\begin{equation*}
\Delta F = \frac{1}{D_1(x)}\Delta_g F - D_1(x)^{\frac{n-2}{2}}\nabla (D_1^{\frac{2-n}{2}})\cdot \nabla F.
\end{equation*}
Thus we may write the equation (\ref{wtt0}) as a Riemannian elliptic equation of the form
\begin{equation}
\Delta_g F + \mbox{``lower-order terms''} = \frac{D_1(x)}{R(x,0)}w_{tt}(x,0) \qquad \text{ a.e. on } K
\end{equation}
where the ``lower-order terms'' contain the first-order and zeroth-order in $F$ with $L^{\infty}(\Omega)$ coefficients since $D_1\in C^2(\Omega)$. For this equation, with the assumption (\ref{positivity1}) and the fact that $D_1(x)$ is compactly supported in $\Omega$, we have the following Carleman estimates for the solution $F$ from (\ref{elliptic}):
\begin{equation*}
\tau\int_{K}e^{2\tau d}|\mathcal{D} F|^2\,dx + \tau^3\int_{K} e^{2\tau d}|F|^2\,dx \leq C\int_{K}e^{2\tau d}|w_{tt}(x,0)|^2\,dx + BT(F)
\end{equation*}
where the boundary terms $BT(F)\equiv 0$ since $\supp F \subset {\rm Int } K$. In addition, note from (\ref{gradientnorm}) we have 
$|\mathcal{D}F|^2 =D_1(x)|\nabla F|^2\geq \frac{1}{c_0}|\nabla F|^2$, thus we get  
\begin{equation}\label{30}
\tau\int_{K}e^{2\tau d}|\nabla F|^2\,dx + \tau^3\int_{K} e^{2\tau d}|F|^2\,dx \leq C\int_{K}e^{2\tau d}w_{tt}^2(x,0)\,dx.
\end{equation}
Then with the assumptions (\ref{regularity2}), (\ref{positivity1}) and the equation (\ref{wtt0}), we also get
\begin{equation}\label{31}
\int_{K}e^{2\tau d}|\Delta F|^2\,dx \leq C\int_{K} e^{2\tau d} w_{tt}^2(x,0)\,dx.     
\end{equation}
Combining (\ref{30}) and (\ref{31}), we get 
$$
\int_{K}e^{2\tau d} \left(|\Delta F|^2 + |\nabla F|^2 + |F|^2\right)\,dx \leq C \int_{K} e^{2\tau d}w_{tt}^2(x,0)\,dx.
$$
In this inequality, the left-hand side does not change if we replace $K$ by $\Omega$ since $\supp F \subset {\rm Int } K$. The right-hand side is bounded by $\displaystyle C \int_{\Omega} e^{2\tau d}w_{tt}^2(x,0)\,dx$. This proves (\ref{ellipticcarleman}). 
\end{proof}

To continue with the proof, note we have
\begin{lemma}
The following identity holds true:
\begin{eqnarray}\label{L1}
\int_{\Omega} e^{2\tau\varphi(x,0)} w_{tt}^2(x,0)\,dx & = & -4c\tau\int_{\Omega}\int_{-T}^0 te^{2\tau\varphi}\left[w_{tt}^2+D_1(x)|\nabla w_t|^2\right]\,dtdx \nonumber \\ 
& \ & - 4\tau\int_{\Omega}\int_{-T}^0 e^{2\tau\varphi}w_{tt}D_1(x)\nabla d(x)\cdot\nabla w_t\,dtdx \nonumber\\
& \ & - 2\tau\int_{\Omega}\int_{-T}^0 e^{2\tau\varphi}w_{tt}\nabla D_1(x)\cdot\nabla w_t\,dtdx \nonumber\\
& \ & + \int_{\Omega}e^{2\tau\varphi(x,-T)}[w_{tt}^2(x,-T)+D_1(x)|\nabla w_t(x, -T)|^2]\,dx\nonumber\\
& \ & + 2\int_{\Omega}\int_{-T}^0e^{2\tau\varphi}w_{tt}\left[2\nabla D_1(x)\cdot\nabla w_t + \Delta D_1(x) w_t\right]\,dtdx\nonumber\\
& \ & + 2\int_{\Omega}\int_{-T}^0e^{2\tau\varphi}w_{tt}\left[\Delta F(x) R_t + 2\nabla F(x)\cdot \nabla R_t + F(x) \Delta R_t\right]\,dtdx. \nonumber\\
& \ &
\end{eqnarray}
\end{lemma}

\begin{proof}
Starting from the left-hand side, we calculate
\begin{eqnarray}\label{L2}
\int_{\Omega} e^{2\tau\varphi(x,0)} w_{tt}^2(x,0)\,dx & = & \int_{\Omega}\int_{-T}^0\frac{d}{dt}(e^{2\tau\varphi}w_{tt}^2)\,dtdx + \int_{\Omega}e^{2\tau\varphi(x,-T)}w_{tt}^2(x,-T)\,dx \nonumber\\[2mm]
& = & - 4c\tau\int_{\Omega}\int_{-T}^0 te^{2\tau\varphi} w_{tt}^2\,dtdx + 2\int_{\Omega}\int_{-T}^0e^{2\tau\varphi}w_{tt}w_{ttt}\,dtdx \nonumber\\
& \ & + \int_{\Omega}e^{2\tau\varphi(x,-T)}w_{tt}^2(x,-T)\,dx.
\end{eqnarray}
Evaluate the second integral term on the right-hand side of (\ref{L2}), use the $w_t$-equation (\ref{wt}) and (\ref{rhs}), we have
\begin{eqnarray}\label{L3}
\int_{\Omega}\int_{-T}^0e^{2\tau\varphi}w_{tt}w_{ttt}\,dtdx & = & \int_{\Omega}\int_{-T}^0e^{2\tau\varphi}w_{tt}\left[D_1(x)\Delta w_t + 2\nabla D_1\cdot\nabla w_t + \Delta D_1(x) w_t\right]\,dtdx \nonumber \\
& \ &  + \int_{\Omega}\int_{-T}^0e^{2\tau\varphi}w_{tt}\left[\Delta F(x) R_t + 2\nabla F(x)\cdot \nabla R_t + F(x) \Delta R_t\right]\,dtdx. \nonumber \\
& \ &
\end{eqnarray}
Next we evaluate the first integral term $\displaystyle\int_{\Omega}\int_{-T}^0e^{2\tau\varphi}w_{tt}D_1(x)\Delta w_t\,dtdx$ on the right-hand side of (\ref{L3}), by using the Green's formula and the vanishing boundary condition $w=0$ on $\Sigma_T$, we get
\begin{eqnarray}\label{L4}
\int_{\Omega}\int_{-T}^0e^{2\tau\varphi}w_{tt}D_1(x)\Delta w_t\,dtdx & = & -\int_{\Omega}\int_{-T}^0\nabla(e^{2\tau\varphi}w_{tt}D_1(x))\cdot\nabla w_t\,dtdx \nonumber\\[2mm]
& = & - 2\tau\int_{\Omega}\int_{-T}^0e^{2\tau\varphi}w_{tt}D_1(x)\nabla d(x)\cdot\nabla w_t\,dtdx \nonumber \\
& \ & -\int_{\Omega}\int_{-T}^0e^{2\tau\varphi}D_1(x)\nabla w_{tt}\cdot\nabla w_t\,dtdx  \nonumber\\
& \ & - \int_{\Omega}\int_{-T}^0e^{2\tau\varphi}w_{tt}\nabla D_1(x)\cdot \nabla w_t\,dtdx.
\end{eqnarray}
Last, we calculate the second integral on the right-hand side of (\ref{L4}), use integration by parts and the zero initial condition $w_t(x,0)=0$, we get
\begin{eqnarray}\label{L5}
\int_{\Omega}\int_{-T}^0e^{2\tau\varphi}D_1(x)\nabla w_{tt}\cdot\nabla w_t\,dtdx & = & \frac{1}{2}\int_{\Omega}\int_{-T}^0e^{2\tau\varphi}D_1(x)\frac{d}{dt}(|\nabla w_t|^2)\,dtdx \nonumber \\[2mm]
& = & -\frac{1}{2}\int_{\Omega}e^{2\tau\varphi(x,-T)}D_1(x)|\nabla w_t(x,-T)|^2\,dx \nonumber\\
& \ & +\frac{1}{2}\int_{\Omega}\int_{-T}^04c\tau te^{2\tau\varphi}D_1(x)|\nabla w_t|^2\,dtdx.
\end{eqnarray}
Combine together (\ref{L2}), (\ref{L3}), (\ref{L4}) and (\ref{L5}), we readily get (\ref{L1}).
\end{proof}
From the identity (\ref{L1}), we may easily get the following estimate
\begin{eqnarray}\label{L6}
&  & \int_{\Omega} e^{2\tau\varphi(x,0)} w_{tt}^2(x,0)\,dx \nonumber \\
& \leq & C\tau\int_{Q_T}e^{2\tau\varphi}[w_{tt}^2+|\mathcal{D} w_t|^2]\,dtdx + C\int_{Q(\sigma)}e^{2\tau\varphi}w_{t}^2\,dtdx + C\int_{Q_T\backslash Q(\sigma)}e^{2\tau\varphi}w_{t}^2\,dtdx \nonumber \\
& \ & + C\int_{\Omega}e^{2\tau\varphi(x,-T)}[w_{tt}^2(x,-T)+|\mathcal{D} w_t(x,-T)|^2]\,dx \nonumber\\
& \ & +C\int_{Q_T}e^{2\tau\varphi}\left[|\Delta F|^2R_t^2 + |\nabla F|^2|\nabla R_t|^2 + |F|^2 (\Delta R_t)^2\right]\,dtdx.
\end{eqnarray}
Apply the Carleman estimate (\ref{th1carleman1}) for the first two terms on the right-hand side of (\ref{L6}), and apply $\varphi\leq \sigma$ on $Q_T\backslash Q(\sigma)$ for the third term on the right-hand side of (\ref{L6}), and also note $\varphi(x,-T) \leq -\delta$ from (\ref{pseudoproperty1}), we have for $\tau$ large enough
\begin{eqnarray}\label{L7}
\int_{\Omega} e^{2\tau\varphi(x,0)} w_{tt}^2(x,0)\,dx & \leq & BT(w_t) + C e^{2\tau\sigma}\int_{Q_T}w_t^2\,dtdx + C\tau^3e^{-2\tau\delta}[E(-T)+E(T)] \nonumber\\
& \ & + C\int_{Q_T}e^{2\tau\varphi}\left[|\Delta F|^2R_t^2 + |\nabla F|^2|\nabla R_t|^2 + |F|^2 (\Delta R_t)^2\right]\,dtdx.
\end{eqnarray}
We now estimate the terms on the right-hand side of (\ref{L7}). First, since $w=0$ on $\Sigma_T$, we may readily have from (\ref{boundary}) and the geometrical assumption (A.1) that 
\begin{equation}\label{L8}
|BT(w_t)|\leq C e^{C\tau}\left\|\partial_\nu w_t\right\|^2_{L^2(\Gamma_1\times(-T,T))}.
\end{equation}
Next, for the middle two terms of (\ref{L7}), we may apply the energy estimates (\ref{energyestimate}) to the $w_t$-equation (\ref{wt}) and use the assumption (\ref{regularity2}) to get
\begin{equation}\label{L9}
\int_{Q_T}w_{t}^2\,dtdx + E(-T) + E(T) \leq C \|F\|_{\Delta}^2.
\end{equation}
Last, to handle the last term of (\ref{L7}), we claim that we have
\begin{eqnarray}\label{L10}
& \ & \int_{Q_T}e^{2\tau\varphi}\left[|\Delta F|^2R_t^2 + |\nabla F|^2|\nabla R_t|^2 + |F|^2 (\Delta R_t)^2\right]\,dtdx \nonumber\\[2mm] 
& \leq & o(1)\int_{\Omega}e^{2\tau\varphi(x,0)}[|\Delta F|^2+|\nabla F|^2+|F|^2]\,dx
\end{eqnarray}
where the term $o(1)$ satisfies $\displaystyle\lim_{\tau\to\infty}o(1)=0$. To see this claim, we estimate the term $\displaystyle\int_{Q_T}e^{2\tau\varphi}|F|^2 (\Delta R_t)^2\,dtdx$ as follows and apply similar estimates to the other terms. Note we have
\begin{eqnarray}\label{L11}
\int_{Q_T}e^{2\tau\varphi}|F|^2 (\Delta R_t)^2\,dtdx & = & \int_{Q_T}e^{2\tau\varphi(x,0)}|F|^2e^{2\tau[\varphi(x,t)-\varphi(x,0)]}(\Delta R_t)^2\,dtdx \nonumber\\
& \leq & \int_{\Omega}e^{2\tau\varphi(x,0)}|F|^2\left(\int_{-T}^Te^{-2c\tau t^2}\left\|\Delta R_t\right\|^2_{L^{\infty}(\Omega)}\,dt\right)dx
\end{eqnarray}
Note $e^{-2c\tau t^2}\to 0$ as $\tau\to\infty$ except at $t=0$, and from assumption (\ref{regularity2}) we have $\Delta R_t\in L^{1}(Q_T)$. Thus by Dominated Convergence Theorem we have 
\begin{equation*}
\int_{-T}^Te^{-2c\tau t^2}\left\|\Delta R_t\right\|^2_{L^{\infty}(\Omega)}\,dt = 0, \ \tau\to\infty.
\end{equation*}
Plug this into (\ref{L11}), and use the similar estimates for other terms we readily get the desired claim (\ref{L10}).

Finally, note that $\varphi(x,0)=d(x)$, and $\displaystyle \min_{x\in\Omega} d(x) > \sigma > 0$ from (\ref{pseudoproperty3}), thus all the terms on the right-hand side of (\ref{L7}), except the boundary terms, can be absorbed into the left-hand side by $e^{2\tau \min d(x)}\|F\|_{\Delta}^2$. Hence we have
\begin{equation*}
\|F\|^2_{\Delta} \leq C e^{C\tau}\left\|\partial_\nu w_t\right\|^2_{L^2(\Gamma_1\times(-T,T))}.
\end{equation*}
Since $w_t$ is an odd function in $t\in(-T, T)$ (recall $w$ is evenly extended to $(-T,0)$), we get the desired stability estimate: There exists a constant $C>0$, which depends on $\tau$ exponentially,    
such that  
\begin{equation*}
\|F\|^2_{\Delta} \leq C\left\|\partial_\nu w_t\right\|^2_{L^2(\Gamma_1\times(0,T))}.
\end{equation*}

\medskip
\noindent{\bf Proof of Theorem~1.2}. Finally, we provide the proof of the stability for the original inverse problem. From the relationship (\ref{relation}) between the original inverse problem and the inverse source problem, this pretty much boils down to verify that the regularity assumption (\ref{regularity}) implies the regularity assumption (\ref{regularity2}) in Theorem~1.3. Note by the regularity theory (\ref{reg}) the assumption (\ref{regularity}) on the initial and boundary conditions $\{f, h\}$ implies the solution of the equation (\ref{1}) $u$ satisfies 
\begin{equation*}
u\in C\left([-T,T]; H^{\gamma}(\Omega)\right).
\end{equation*} 
As $\gamma>\frac{n}{2}+3$, we have the following embedding $H^{\gamma}(\Omega)\hookrightarrow W^{3,\infty}(\Omega)$ and hence the regularity assumption (\ref{regularity}) implies the corresponding regularity assumption (\ref{regularity2}) for the inverse source problem. This completes the proof.

\bigskip 
\noindent\textbf{Proof of Corollary~\ref{thm:PAT}}. 
Recall the PAT measurements induced by $(c,p_0)$ and $(\tilde{c},\tilde{p}_0)$ are $p|_{\Sigma}$ and $\tilde{p}|_{\Sigma}$, respectively. We will use the notations introduced in~\eqref{eq:variable}. 
Suppose $p_{0} c^{-2} = \tilde{p}_{0} \tilde{c}^{-2}$, that is, $f=\tilde{f}$. By Theorem~\ref{Th1}, we conclude $D=\tilde{D}$ hence $c=\tilde{c}$. It follows from the a-priori relation that $p_{0} = \tilde{p}_{0}$. This completes the proof of the corollary.



\section*{Acknowledgement}
The research of G. Uhlmann is partly supported by NSF. The research of Y. Yang is partially supported by the NSF grants DMS-2006881, DMS-2237534, DMS-2220373, and the NIH grant R03-EB033521.


\bibliographystyle{abbrv}
\bibliography{refs}

\begin{thebibliography}{10}

\bibitem{acosta2015multiwave}
S.~Acosta and C.~Montalto.
\newblock Multiwave imaging in an enclosure with variable wave speed.
\newblock {\em Inverse Problems}, 31(6):065009, 2015.

\bibitem{acosta2018thermoacoustic}
S.~Acosta and B.~Palacios.
\newblock Thermoacoustic tomography for an integro-differential wave equation
  modeling attenuation.
\newblock {\em Journal of Differential Equations}, 264(3):1984--2010, 2018.

\bibitem{agranovsky2009on}
M.~{Agranovsky}, P.~{Kuchment}, and L.~{Kunyansky}.
\newblock On reconstruction formulas and algorithms for the thermoacoustic
  tomography.
\newblock {\em Photoacoustic Imaging and Spectroscopy}, pages 89--101, 2009.

\bibitem{B2004}
M.~Bellassoued.
\newblock Uniqueness and stability in determining the speed of propagation of
  second-order hyperbolic equation with variable coefficients.
\newblock {\em Applicable Analysis}, 83(10):983--1014, 2004.

\bibitem{B-Y2017}
M.~Bellassoued and M.~Yamamoto.
\newblock {\em Carleman Estimates and Applications to Inverse Problems for
  Hyperbolic Systems}, volume~8.
\newblock Springer, 2017.

\bibitem{BCIY2001}
A.~Bukhgeim, J.~Cheng, V.~Isakov, and M.~Yamamoto.
\newblock Uniqueness in determining damping coefficients in hyperbolic
  equations.
\newblock In {\em Analytic Extension Formulas and their Applications}, pages
  27--46. Springer, 2001.

\bibitem{BK1981}
A.~L. Bukhgeim and M.~V. Klibanov.
\newblock Global uniqueness of a class of multidimensional inverse problems.
\newblock In {\em Doklady Akademii Nauk}, volume 260, pages 269--272. Russian
  Academy of Sciences, 1981.

\bibitem{chervova2016time}
O.~Chervova and L.~Oksanen.
\newblock Time reversal method with stabilizing boundary conditions for
  photoacoustic tomography.
\newblock {\em Inverse Problems}, 32(12):125004, 2016.

\bibitem{cox2007photoacoustic}
B.~T. Cox, S.~R. Arridge, and P.~C. Beard.
\newblock Photoacoustic tomography with a limited-aperture planar sensor and a
  reverberant cavity.
\newblock {\em Inverse Problems}, 23(6):S95, 2007.

\bibitem{finch2004determining}
D.~{Finch} and R.~S.~K. {Patch}.
\newblock Determining a function from its mean values over a family of spheres.
\newblock {\em SIAM Journal on Mathematical Analysis}, 35(5):1213--1240, 2004.

\bibitem{finch2009recovering}
D.~{Finch}, {Rakesh}, and L.~{Kunyansky}.
\newblock Recovering a function from its spherical mean values in two and three
  dimensions.
\newblock {\em Photoacoustic Imaging and Spectroscopy}, 2009.

\bibitem{haltmeier2005filtered}
M.~{Haltmeier}, T.~{Schuster}, and O.~{Scherzer}.
\newblock Filtered backprojection for thermoacoustic computed tomography in
  spherical geometry.
\newblock {\em Mathematical Methods in The Applied Sciences},
  28(16):1919--1937, 2005.

\bibitem{homan2013multiwave}
A.~Homan.
\newblock Multi-wave imaging in attenuating media.
\newblock {\em Inverse Problems and Imaging}, 7(4):1235--1250, 2013.

\bibitem{H1963}
L.~H{\"o}rmander.
\newblock {\em The Analysis of Linear Partial Differential Pperators III:
  Pseudo-differential Operators}.
\newblock Springer Science \& Business Media, 2007.

\bibitem{hristova2009time}
Y.~{Hristova}.
\newblock Time reversal in thermoacoustic tomography—an error estimate.
\newblock {\em Inverse Problems}, 25(5):55008, 2009.

\bibitem{hristova2008reconstruction}
Y.~{Hristova}, P.~{Kuchment}, and L.~{Nguyen}.
\newblock Reconstruction and time reversal in thermoacoustic tomography in
  acoustically homogeneous and inhomogeneous media.
\newblock {\em Inverse Problems}, 24(5):55006, 2008.

\bibitem{huang2023piecewise}
G.~Huang, J.~Qian, and Y.~Yang.
\newblock Piecewise acoustic source imaging with unknown speed of sound using a
  level-set method.
\newblock {\em Communications on Applied Mathematics and Computation}, pages
  1--26, 2023.

\bibitem{HIY2020}
X.~Huang, O.~Y. Imanuvilov, and M.~Yamamoto.
\newblock Stability for inverse source problems by carleman estimates.
\newblock {\em Inverse Problems}, 36(12):125006, 2020.

\bibitem{Imanuvilov2001a}
O.~Y. Imanuvilov and M.~Yamamoto.
\newblock Global uniqueness and stability in determining coefficients of wave
  equations.
\newblock {\em Communications in Partial Differential Equations},
  26(7-8):1409--1425, 2001.

\bibitem{Imanuvilov2003}
O.~Y. Imanuvilov and M.~Yamamoto.
\newblock Determination of a coefficient in an acoustic equation with a single
  measurement.
\newblock {\em Inverse Problems}, 19(1):157, 2003.

\bibitem{Isakov2006}
V.~Isakov.
\newblock {\em Inverse Problems for Partial Differential Equations}.
\newblock Springer, 2018.

\bibitem{Jellali2006}
D.~Jellali.
\newblock An inverse problem for the acoustic wave equation with finite sets of
  boundary data.
\newblock {\em Journal of Inverse \& Ill-Posed Problems}, 14(7), 2006.

\bibitem{kian2023determination}
Y.~Kian and G.~Uhlmann.
\newblock Determination of the sound speed and an initial source in
  photoacoustic tomography.
\newblock {\em arXiv preprint arXiv:2302.03457}, 2023.

\bibitem{Klibanov1992}
M.~V. Klibanov.
\newblock Inverse problems and carleman estimates.
\newblock {\em Inverse problems}, 8(4):575, 1992.

\bibitem{Klibanov2021}
M.~V. Klibanov and J.~Li.
\newblock {\em Inverse Problems and Carleman Estimates: Global Uniqueness,
  Global Convergence and Experimental Data}, volume~63.
\newblock Walter de Gruyter GmbH \& Co KG, 2021.

\bibitem{Klibanov2004}
M.~V. Klibanov and A.~A. Timonov.
\newblock {\em Carleman Estimates for Coefficient Inverse Problems and
  Numerical Applications}, volume~46.
\newblock Walter de Gruyter, 2004.

\bibitem{KY2006}
M.~V. Klibanov and M.~Yamamoto.
\newblock Lipschitz stability of an inverse problem for an acoustic equation.
\newblock {\em Applicable Analysis}, 85(05):515--538, 2006.

\bibitem{knox2018jointdetermination}
C.~Knox and A.~Moradifam.
\newblock Determining both the source of a wave and its speed in a medium from
  boundary measurements.
\newblock {\em Inverse Problems}, 36(2):025002, 2020.

\bibitem{kuchment2008mathematics}
P.~{Kuchment} and L.~{Kunyansky}.
\newblock Mathematics of thermoacoustic tomography.
\newblock {\em European Journal of Applied Mathematics}, 19(2):191--224, 2008.

\bibitem{kunyansky2007explicit}
L.~A. Kunyansky.
\newblock Explicit inversion formulae for the spherical mean radon transform.
\newblock {\em Inverse problems}, 23(1):373, 2007.

\bibitem{LLT1986}
I.~Lasiecka, J.-L. Lions, and R.~Triggiani.
\newblock Non homogeneous boundary value problems for second order hyperbolic
  operators.
\newblock {\em Journal de Math{\'e}matiques pures et Appliqu{\'e}es},
  65(2):149--192, 1986.

\bibitem{LRS1986}
M.~M. Lavrentev, V.~G. Romanov, and S.~P. Shishatskii.
\newblock {\em Ill-posed Problems of Mathematical Physics and Analysis},
  volume~64.
\newblock American Mathematical Soc., 1986.

\bibitem{LM1972}
J.~L. Lions and E.~Magenes.
\newblock {\em Non-homogeneous Boundary Value Problems and Applications: Vol.
  1}, volume 181.
\newblock Springer Science \& Business Media, 2012.

\bibitem{liu2015jointdetermination}
H.~Liu and G.~Uhlmann.
\newblock Determining both sound speed and internal source in thermo-and
  photo-acoustic tomography.
\newblock {\em Inverse Problems}, 31(10):105005, 2015.

\bibitem{Liu2013}
S.~Liu.
\newblock Recovery of the sound speed and initial displacement for the wave
  equation by means of a single dirichlet boundary measurement.
\newblock {\em Evolution Equations \& Control Theory}, 2(2), 2013.

\bibitem{LPS2023}
S.~Liu, A.~Pierrottet, and S.~Scruggs.
\newblock Recover all coefficients in second-order hyperbolic equations from
  finite sets of boundary measurements.
\newblock {\em Inverse Problems}, 39(10):105009, 2023.

\bibitem{LiuT2011-2}
S.~Liu and R.~Triggiani.
\newblock Global uniqueness and stability in determining the damping
  coefficient of an inverse hyperbolic problem with nonhomogeneous neumann bc
  through an additional dirichlet boundary trace.
\newblock {\em SIAM journal on mathematical analysis}, 43(4):1631--1666, 2011.

\bibitem{Liu2012}
S.~Liu and R.~Triggiani.
\newblock Global uniqueness and stability in determining the damping
  coefficient of an inverse hyperbolic problem with non-homogeneous dirichlet
  bc through an additional localized neumann boundary trace.
\newblock {\em Applicable Analysis}, 91(8):1551--1581, 2012.

\bibitem{Liu-T2013}
S.~Liu and R.~Triggiani.
\newblock Boundary control and boundary inverse theory for non-homogeneous
  second-order hyperbolic equations: a common carleman estimates approach.
\newblock {\em HCDTE Lecture notes, AIMS Book Series on Applied Mathematics},
  6:227--343, 2013.

\bibitem{nguyen2016dissipative}
L.~V. Nguyen and L.~A. Kunyansky.
\newblock A dissipative time reversal technique for photoacoustic tomography in
  a cavity.
\newblock {\em SIAM Journal on Imaging Sciences}, 9(2):748--769, 2016.

\bibitem{palacios2016reconstruction}
B.~Palacios.
\newblock Reconstruction for multi-wave imaging in attenuating media with large
  damping coefficient.
\newblock {\em Inverse Problems}, 32(12):125008, 2016.

\bibitem{qian2011an}
J.~{Qian}, P.~D. {Stefanov}, G.~{Uhlmann}, and H.~{Zhao}.
\newblock An efficient neumann series-based algorithm for thermoacoustic and
  photoacoustic tomography with variable sound speed.
\newblock {\em SIAM Journal on Imaging Sciences}, 4(3):850--883, 2011.

\bibitem{Stefanov2009variable}
P.~Stefanov and G.~Uhlmann.
\newblock Thermoacoustic tomography with variable sound speed.
\newblock {\em Inverse Problems}, 25(7):075011, 2009.

\bibitem{Stefanov2011brain}
P.~Stefanov and G.~Uhlmann.
\newblock Thermoacoustic tomography arising in brain imaging.
\newblock {\em Inverse Problems}, 27(4):045004, 2011.

\bibitem{stefanov2013instability}
P.~Stefanov and G.~Uhlmann.
\newblock Instability of the linearized problem in multiwave tomography of
  recovery both the source and the speed.
\newblock {\em Inverse Problems \& Imaging}, 7(4), 2013.

\bibitem{stefanov2013recovery}
P.~Stefanov and G.~Uhlmann.
\newblock Recovery of a source term or a speed with one measurement and
  applications.
\newblock {\em Transactions of the American Mathematical Society},
  365(11):5737--5758, 2013.

\bibitem{stefanov2015multiwave}
P.~Stefanov and Y.~Yang.
\newblock Multiwave tomography in a closed domain: averaged sharp time
  reversal.
\newblock {\em Inverse Problems}, 31(6):065007, 2015.

\bibitem{stefanov2017multiwave}
P.~Stefanov and Y.~Yang.
\newblock Multiwave tomography with reflectors: Landweber’s iteration.
\newblock {\em Inverse Problems and Imaging}, 11(2):373--401, 2017.

\bibitem{Stefanov2017planar}
P.~Stefanov and Y.~Yang.
\newblock Thermo and photoacoustic tomography with variable speed and planar
  detectors.
\newblock {\em SIAM Journal of Mathematical Analysis}, 49(1):297--310, 2017.

\bibitem{TX2007}
R.~Triggiani and X.~Xu.
\newblock Pointwise carleman estimates, global uniqueness, observability, and
  stabilization for schrodinger equations on the riemannian manifolds at the h1
  (o)-level.
\newblock {\em Contemporary Mathematics}, 426:339, 2007.

\bibitem{TY2002}
R.~Triggiani and P.~Yao.
\newblock Carleman estimates with no lower-order terms for general riemann wave
  equations. global uniqueness and observability in one shot.
\newblock {\em Applied Mathematics \& Optimization}, 46:331--375, 2002.

\bibitem{Y1999}
M.~Yamamoto.
\newblock Uniqueness and stability in multidimensional hyperbolic inverse
  problems.
\newblock {\em Journal de math{\'e}matiques pures et appliqu{\'e}es},
  78(1):65--98, 1999.

\end{thebibliography}

\end{document}